\documentclass[runningheads,a4paper]{llncs}

\usepackage{amssymb}
\usepackage{graphicx}

\usepackage{setspace}
\usepackage{amsmath}
\usepackage{parskip}
\usepackage{hyperref}
\usepackage{color}

\newcommand{\beq}[1]{\begin{equation}\label{#1}}
\newcommand{\eeq}{\end{equation}}

\def\a{\alpha}  \def\d{\delta} 
    \def\g{\gamma}
 \def\i{\iota} \def\k{\kappa}

\def\t{\tau} \def\om{\omega}

\newcommand{\brac}[1]{\left(#1\right)}

\newcommand{\bfrac}[2]{\left(\frac{#1}{#2}\right)}

\def\E{\mbox{{\bf E}}}
\def\Pr{\mbox{{\bf Pr}}}
\def\whp{{\bf whp}}

\newcommand{\ignore}[1]{}

\def\uar{{\bf uar}}

\usepackage{url}
\newcommand{\keywords}[1]{\par\addvspace\baselineskip
\noindent\keywordname\enspace\ignorespaces#1}

\def\PA{\text{PA}_t(m,\d)}
\newcommand*\samethanks[1][\value{footnote}]{\footnotemark[#1]}

\begin{document}

\mainmatter  

\title{Local majority dynamics on preferential attachment graphs}
\titlerunning{Local majority dynamics on preferential attachment graphs}

\author{Mohammed Amin Abdullah\inst{1}\thanks{Research supported by the EPSRC Grant No. EP/K019749/1.}\and Michel Bode\inst{2} \and Nikolaos Fountoulakis\inst{3}\samethanks}

\authorrunning{M. A. Abdullah, M. Bode and N. Fountoulakis}

\institute{Mathematical and Algorithmic Sciences Lab, Huawei Technologies Ltd,\\
\email{mohammed.abdullah@huawei.com}\\
\and School of Mathematics, University of Birmingham,\\
\email{michel.bode@gmx.de}
\and
School of Mathematics, University of Birmingham,\\
\email{n.fountoulakis@bham.ac.uk}
}
%
%

\maketitle

\begin{abstract}

Suppose in a graph $G$ vertices can be either red or blue. Let $k$ be odd. At each time step, each vertex $v$ in $G$ polls $k$ random neighbours and takes the majority colour. If it doesn't have $k$ neighbours, it simply polls all of them, or all less one if the degree of $v$ is even. We study this protocol on the preferential attachment model of Albert and Barab\'asi  \cite{ar:StatMechs},  which gives rise to a degree distribution that has roughly power-law $P(x) \sim \frac{1}{x^{3}}$, as well as generalisations which give exponents larger than $3$. The setting is as follows: Initially each vertex of $G$ is red independently with probability $\alpha < \frac{1}{2}$, and is otherwise blue. We show that if $\alpha$ is sufficiently biased away from $\frac{1}{2}$, then with high probability, consensus is reached on the initial global majority within $O(\log_d \log_d t)$ steps.  Here $t$ is the number of vertices and $d \geq 5$ is the minimum of $k$ and $m$ (or $m-1$ if $m$ is even), $m$ being the number of edges each new vertex adds in the preferential attachment generative process. Additionally, our analysis  reduces the required bias of $\a$ for graphs of a given degree sequence studied in \cite{LocalMaj} (which includes, e.g., random regular graphs). 
\end{abstract}

\section{Introduction}
Let  $G=(V,E)$ be a graph where each vertex maintains an opinion,  which we speak of in terms of two colours - red and blue.  We make no assumptions about the properties of the colours/opinions except that vertices can distinguish between them. We are interested in distributed protocols on $G$ that can bring about consensus to a single opinion. 

One of the simplest and most widely studied distributed consensus algorithms is the \emph{voter model} (see, e.g., \cite[ch.~14]{AlFi}). In the discrete time setting, at each time step $\t$, each vertex chooses a single neighbour uniformly at random (\uar) and assumes its opinion.  The number of different opinions in the system is clearly non-increasing, and consensus is reached almost surely in finite, non-bipartite, connected graphs. Using an elegant martingale argument, \cite{Peleg} determined the probability of consensus to a particular colour. In our context this would be the sum of the degrees of vertices which started with that colour, as a fraction of the sum of degrees over all vertices. Thus, on regular graphs, for example, if the initial proportion of reds is a constant $\alpha$, the probability of a red consensus is $\alpha$. This probability is increased on non-regular graphs if the minority is ``privileged'' by sitting on high degree vertices (as in say, for example, the small proportion of high degree vertices  in a graph with power-law distribution). This motivates an alternative where the majority is certain, or highly likely, to win.

The \emph{local majority} protocol in a synchronous discrete time setting does the following: At each time step, each vertex $v$ polls all its neighbours and assumes the majority colour in the next time step. This can be motivated by both a prescriptive and a descriptive view. In the former, as a consensus protocol, it can be seen as a distributed co-ordination mechanism for networked systems. In the latter, it can be seen as a natural process occurring, for example in social networks where it may represent the spread of influence. 

Let $k$ be odd. Suppose at time step $\t=0$ each vertex of a graph $G=(V,E)$ is either red or blue. In this paper we study the following generalisation of the local majority protocol (also in a synchronous, discrete time setting):

\begin{definition}[$k$-choice Local Majority Protocol $\mathcal{MP}^k$]
For each vertex $v \in V$, for each time step $\t=1,2,\ldots$ do the following: choose a set of $k$ neighbours of $v$ uniformly at random. The colour of $v$ at time step $\t$ is the majority colour
of this set at time step $\t-1$. If $v$ does not have $k$ neighbours, then choose a random set of largest possible odd cardinality.
\end{definition}

Clearly, we can retrieve the local majority protocol by setting $k$ to be the maximum degree, for example.

In addition to which colour dominates, one is also interested in how long it takes to reach consensus. In the voter model, there is a duality between the voting process and multiple random walks on the graph. The time it takes for a single opinion to emerge is the same as the time it takes for $n$ independent random walks - one starting at each vertex - to coalesce into a single walk, where two or more random walks coalesce if they are on the same vertex at the same time. Thus, consensus time can be determined by studying this multiple walk process. However, the analyses of local-majority-type protocols have not been readily amenable to the established techniques for the voter model, namely, martingales and coalescing random walks. Martingales have proved elusive and the random walks duality does not readily transfer, nor is there an obvious way of altering the walks appropriately. Thus, ad-hoc techniques and approaches have been developed. 

We say a sequence of events $(\mathcal{E}_t)_t$ occurs \emph{with high probability} (\whp) if $\Pr(\mathcal{E}_t)\rightarrow 1$ as $t \rightarrow \infty$. In this paper, the underlying parameter $t$ which goes to infinity will be the number of vertices in the sequence of graphs $\text{PA}_t(m,\d)$ we consider.

The main result in this paper will be to show that when each vertex of a preferential attachment graph $\text{PA}_t(m,\d)$ (introduced in the next section) is red independently with probability $\a<1/2$, where $\a$ is sufficiently biased away from $1/2$, then the system will converge to the majority colour with high probability, and we give an upper bound for the number of steps this takes.

\section{Preferential attachment graphs}
The preferential attachment models have their origins in the work of Yule \cite{Yule}, where a growing model is proposed in the context 
of the evolution of species. A similar model was proposed by Simon \cite{Simon} in the statistics of language. 
The principle of these models was used by Albert and Barab\'asi  \cite{ar:StatMechs} to describe a random graph model where vertices arrive 
one by one and each of them throws a number of half-edges to the existing graph. Each half-edge is connected to a vertex with probability
that is proportional to the degree of the latter. This model was defined rigorously by Bollob\'as, Riordan, Spencer and 
Tusn\'ady~\cite{DegSeq} (see also~\cite{Diam}). We will describe the most general form of the model which is essentially due to 
Dorogovtsev et al.~\cite{Dor} and Drinea et al.~\cite{Drinea}. Our description and notation below follows that 
from the book of van der Hofstad~\cite{Remco}.

The random graph $\text{PA}_t(m,\d)=(V,E)$ where $V=[t]$ is parameterised by two constants: $m \in \mathbb{N}$, and $\d \in \mathbb{R}$, $\d > -m$.  
It gives rise to a random graph sequence (i.e., a sequence in which each member is a random graph), denoted by
$\brac{\text{PA}_t(m,\d)}_{t=1}^\infty$. The $t$th term of the sequence, $\text{PA}_t(m,\d)$ is a graph with $t$ vertices and $mt$
edges. Further, $\text{PA}_t(m,\d)$ is a subgraph of $\text{PA}_{t+1}(m,\d)$. We define $\text{PA}_t(1,\d)$ first, then use it to define
the general model $\text{PA}_t(m,\d)$ (the Barab\'asi-Albert model corresponds to the case $\d = 0$).

The random graph $\text{PA}_1(1,\d)$ consists of a single vertex with one self-loop. We denote the vertices of $\text{PA}_t(1,\d)$ by
$\{v_1^{(1)}, v_2^{(1)}, \ldots, v_t^{(1)}\}$. We denote the degree of vertex $v_i^{(1)}$ in $\text{PA}_t(1,\d)$ by $D_i(t)$. Then,
conditionally on $\text{PA}_t(1,\d)$, the growth rule to obtain $\text{PA}_{t+1}(1,\d)$ is as follows: We add a single vertex
$v_{t+1}^{(1)}$ having a single edge. The other end of the edge connects to $v_{t+1}^{(1)}$ itself with probability
$\frac{1+\d}{t(2+\d)+(1+\d)}$, and connects to a vertex $v_i^{(1)} \in \text{PA}_t (1,\d)$ with probability
$\frac{D_i(t)+\d}{t(2+\d)+(1+\d)}$ -- we write $v^{(1)}_{t+1}\rightarrow v_i^{(1)}$.  For any $t \in \mathbb{N}$, let $[t]=\{1,\ldots, t \}$. Thus,

\[ \Pr\brac{v^{(1)}_{t+1}\rightarrow v_i^{(1)} \mid \text{PA}_t(1,\d)} = \left\{ 
  \begin{array}{l l}
    \frac{1+\d}{t(2+\d)+(1+\d)} & \quad \text{for $i=t+1$,}\\
    \frac{D_i(t)+\d}{t(2+\d)+(1+\d)} & \quad \text{for $i \in [t]$}
  \end{array} \right.\]
The model $\text{PA}_t(m,\d)$, $m>1$, with vertices $\{ 1,\ldots, t \}$ is derived from
$\text{PA}_{mt}(1,\d/m)$ with vertices $\{v_1^{(1)}, v_2^{(1)}, \ldots, v_{mt}^{(1)}\}$ as follows: For each $i=1, 2, \ldots, t$, we 
 contract the vertices $\{v_{(i-1)+1}^{(1)}, v_{(i-1)+2}^{(1)}, \ldots, v_{(i-1)+t}^{(1)}\}$ into one super-vertex, and identify this
 super-vertex as $i$ in $\text{PA}_t(m,\d)$. When a contraction takes place, all loops and multiple edges are retained. Edges
shared between a set of contracted vertices become loops in the contracted super-vertex. Thus, $\text{PA}_t(m,\d)$ is a graph on $[t]$.

The above process gives a graph whose degree distribution follows a power law with exponent $3+ \delta /m$. This was suggested by 
the analyses in ~\cite{Dor} and~\cite{Drinea}. It was proved rigorously for integral $\delta$ by Buckley and 
Osthus~\cite{BuckleyOsthus2004}.
For a full proof for real $\delta$ see~\cite{Remco}. 
In particular, when $-m < \delta < 0$, the exponent is between 2 and 3. Experimental evidence has shown that this is the case for several
networks that emerge in applications (cf.~\cite{ar:StatMechs}). 
Furthermore, when $m\geq 2$, then $\text{PA}_t(m,\d)$ is \whp\ connected, but when 
$m=1$ this is not the case, giving rise to a logarithmic number of components (see~\cite{Remco}).

\section{Results and related work}
Our main result is the following:
\begin{theorem}\label{MainTheorem}
Let $k \geq 5$ be odd and let $d=m \wedge k$ if $m$ is odd and $d=(m-1) \wedge k$ if $m$ is even. Let $\a^*$ be the smallest positive solution for $x$ in the equation $\Pr\brac{\text{Bin}(d-1,x) \geq \frac{d-1}{2}}=x$. If $\d \geq 0$ and each vertex in $\PA$ is red independently with probability $\a<\a^*$, then for any constant $\varepsilon>0$,  \whp\ under $\mathcal{MP}^k$ every vertex in $\PA$ is blue at all time steps $\t \geq \frac{1+\varepsilon}{\log_d\bfrac{d-1}{2}} \log_d \log_d t$. 
\end{theorem}
Note that $\d = 0$ gives the model proposed in the seminal work of  Albert and Barab\'asi  \cite{ar:StatMechs}, giving power law exponent $3$, and that $\d>0$ gives exponents larger than $3$. We refer the reader to \cite{Remco} for further details.
 
Note, for  $d=5,7,9,11$ we have $\a^*=0.232, 0.347, 0.396, 0.421$ to 3 significant figures (s.f.), respectively.  Of course, $\a^* \rightarrow \frac{1}{2}$ as $d \rightarrow \infty$.

The most closely related work is \cite{LocalMaj}. Here, the same protocol was studied on random graphs of a given degree sequence (which includes random regular graphs) and Erd\H{o}s--R\'{e}nyi random graphs slightly above the connectivity threshold. Similar results similar to Theorem \ref{MainTheorem} were obtained, and in the case of the former model, an almost matching lower bound was shown. It should be noted that the thresholds for $\a$ obtained in this work apply equally to the models in \cite{LocalMaj}, and improve the thresholds for $\a$. To contrast, in that paper, the thresholds for $d=5,7,9,11$ were $0.092, 0.182, 0.234, 0.268$ to 3 s.f., respectively. 

 In \cite{Mossel}, (full) local majority dynamics on $d$-regular $\lambda$-expanders on $n$ vertices is studied. In our notation, they show that when $\alpha \leq 1/2-\frac{2\lambda n}{d}$, there is convergence to the initial majority, so long as $\frac{\lambda}{d} \leq \frac{3}{16}$. Since $\lambda \geq (1-o(1))\sqrt{d}$ for a $d$-regular graph, this condition implies $d \geq 29$. In contrast, our results apply for $d \geq 5$.

In \cite{Cruise} a variant of local majority  is studied where a vertex contacts $m$ others and if $d$ of them have the same colour, the vertex subsequently assumes this colour.  They demonstrate convergence time of $O(\log n)$ and error probability -- the probability of converging on the initial minority -- decaying exponentially with $n$. However, the analysis is done only for the complete graph; our analysis of sparse graphs is a crucial difference, because the techniques employed for complete graphs do not carry through to sparse graphs, nor are they easily adapted. The error probability we give is not as strong but still strong, nevertheless. Furthermore, the convergence time we give is much smaller. 

In \cite{ColinVoting} the authors study the following protocol on random regular graphs and regular expanders:  Each vertex picks two neighbours at random, and takes the majority of these with itself. They show convergence to the initial majority in $O(\log n)$ steps with high probability, subject to sufficiently large initial bias and high enough vertex degree. However, in their setting, the placement of colours can be made adversarially.

In summary, our contribution is demonstrating convergence and time of convergence to initial majority for a generalisation of local majority dynamics for preferential attachment graphs with power-law exponent $3$ and above. As far as we know this is the only such result for power-law graphs (by preferential attachment or otherwise). Furthermore, we have improved the bias thesholds for graphs
of a given degree sequence studied by the first author in \cite{LocalMaj}, which, to the best of our knowledge, were already the best or only known results for small degree graphs in this class (which includes, e.g., random regular graphs).

\section{Structural results}
Throughout this paper we let $\g=\g(m,\d)=\frac{1}{2+\d/m}$. Observe the condition $\d>-m$ (which must be imposed), implies $0<\g<1$. 

Furthermore, for two non-negative functions $f(t), g(t)$ on $\mathbb{N}$ we write $f(t) \lesssim g(t)$ to denote that $f(t) = O(g(t))$.  The underlying asymptotic variable will
 always be $t$, the number of vertices in $\text{PA}_t(m,\d)$.

Let $A$ be a large constant and let
\begin{equation*}
\om = A\log \log t.
\end{equation*}
Let
\begin{equation*}
\k= (\log t)^{7\om}
\end{equation*} 
and define as the \emph{inner core} the vertices $[\k]$, and refer to them as \emph{heavy} vertices. We also refer to vertices outside the inner core as  \emph{light} vertices.

Let
\begin{equation*}
\k_o= (\log t)^{999\om}
\end{equation*}
and define as the \emph{outer core} the vertices $[\k_o]$. 

Call a path \emph{short} if it has length at most $\om$.  Call a cycle \emph{short} if it has at most $2\om+1$ vertices. Here ``cycle'' includes a pair of vertices connected by parallel edges and a vertex with a self-loop.
 
Below, we repeatedly apply the following, which is proved in \cite{PrefPerc} (and for the case $k=1$ was given in \cite{Remco}): 

\begin{proposition}\label{structureProp}
Suppose $i_1, j_1, i_2, j_2, \ldots, i_k, j_k$ are vertices in $\PA$ where $i_s<j_s$ for $s=1,2,\ldots,k$. Then
\begin{equation*}
\Pr(j_1\rightarrow i_2 \cap j_2 \rightarrow i_2, \ldots, j_k \rightarrow i_k)
 \leq M^k\frac{1}{i_1^\g j_1^{1-\g}}\frac{1}{i_2^\g j_2^{1-\g}}\ldots \frac{1}{i_k^\g j_k^{1-\g}}
\end{equation*}
where $M=M(m,\d)$ is a constant that depends only on $m$ and $\d$. 
\end{proposition} 

Below, we also use the fact that $ \frac{1}{i^\g j^{1-\g}} \leq \frac{1}{(ij)^{\frac{1}{2}}}$ when $\d \geq0$ and $i\leq j$. A similar counting approach was used in \cite{ColinCovertime}.

For a vertex $v$ define $\mathcal{B}(v,r)$ to be the $r$ ball of $v$ in $\PA$, the subgraph within distance $r$. 

\begin{lemma}\label{CPCLemma}
With high probability, every vertex $v>\k$ has the following property: $\mathcal{B}(v,\om)$ contains at most one cycle consisting entirely of light vertices.
\end{lemma}
\begin{proof}
Define a \emph{cycle-path-cycle} (CPC) structure as a pair of cycles connected by a path. We consider CPC structures where the cycles and paths are short, that is, cycles have sizes $1 \leq r,s \leq 2\om+1$, and the path has length $0 \leq \ell \leq \om$. Note $r=1$ denotes a self-loop and $r=2$ denotes a pair of parallel edges between two vertices.

We denote by $a_1,\ldots,a_r$ and $b_1, \ldots, b_s$ the vertices of the cycles, and $c_0, \ldots, c_\ell$ the vertices of the path. Without
loss of generality, we may assume $a_1=c_0$ and $b_1=c_\ell$. Thus, the structure has $r+s+\ell$ edges and $r+s+\ell-1$ vertices. 

Applying Proposition \ref{structureProp},  the expected number of such structures lying entirely in $[t]\setminus[\k]$ is bounded by
\begin{align*}
& \sum_{r=1}^{2\om+1}\sum_{\ell=0}^{\om-1}\sum_{s=1}^{2\om+1}\sum_{\k< a_1, \ldots, a_r}\sum_{\k< b_1, \ldots, b_s}\sum_{\k<c_1, \ldots, c_{\ell-1}}\frac{M^{r+s+\ell}}{(a_1b_1)^{3/2}}\prod_{i=2}^r\frac{1}{a_i}\prod_{j=2}^s\frac{1}{b_j}\prod_{k=1}^{\ell-1}\frac{1}{c_k}\\
&\lesssim \brac{\int_\k^t x^{-3/2}\,\mathrm{d}x}^2 \, \sum_{r=1}^{2\om+1}\sum_{\ell=0}^{\om-1}\sum_{s=1}^{2\om+1}(M\log t)^{r+s+\ell}\\
& \lesssim \frac{(\log t)^{6\omega}}{\k}=o(1).
\end{align*}
The rest of the proof is of a similar nature and is continued in the appendix.
\qed
\end{proof}

\begin{lemma}\label{typicalStructureLemma}
With high probability, the following hold for all $v>\k_0$:   
\begin{description}
\item[(i)] $v$ has at most $2$ edges on short paths into $[\k]$. 
\item[(ii)] If $v$ is on a short light cycle, then $v$ has no edge that is on a short (light) path into $[\k]$ but that is not part of the cycle.
\item[(iii)] If $v$ is connected to a short light cycle $C$ by a short light path $P$, then $v$ has at most one edge $e$ such that $e$ is on a short path into $[\k]$ but $e \notin P$.
\end{description}
\end{lemma}
\begin{proof}
\textbf{(i)} Suppose $v$ has three edges $e_1, e_2, e_3$ (possibly parallel) on short paths to $[\k]$ to vertices $i_1, i_2, i_3 \in [\k]$ (not necessarily distinct).  
Then there is a minimal structure $S$ which contains $v, e_1, e_2, e_3, i_1, i_2, i_3$, and a short path from $v$ to $[\k]$ via each edge $e_1, e_2, e_3$. Since $S$ is minimal, there are $0 \leq r \leq 3(\om-1)$ light vertices $a_1, \ldots, a_r$ in $S$ which form the short paths from $v$ to $[\k]$ via $e_1, e_2, e_3$. Also, since $S$ is minimal, it contains at most $3\om$ edges. To consider two extremes, for example, $S$ might be three non-intersecting paths, or a single path with $e_1, e_2, e_3$ being parallel and all other vertices connected connecting by non-parallel edges. 

Observe that each vertex $a_i$ has at least two edges, meaning in the application of Proposition \ref{structureProp} it incurs a fraction $\frac{1}{a_i}$ or less. Applying Proposition \ref{structureProp}, the expected number of structures $S$ is asymptotically bounded from above by
\begin{equation*}
\k^3 M^{3\om}\sum_{r=0}^{3\om}\sum_{\k<a_1, \ldots, a_r}\frac{1}{v^{3/2}}\prod_{i=1}^r\frac{1}{a_i} \lesssim \frac{3\om\k^3 (M\log t)^{3\om}}{v^{3/2}}.
\end{equation*}
Hence, taken over all $v>\k_o$, this is $O\bfrac{3\om\k^3 (M\log t)^{3\om}}{\k_o^{1/2}}=o(1)$. 

The rest of the proof is of a similar nature and is continued in the appendix.
\qed
\end{proof}

We define the \emph{truncated $r$-ball around $v$}, denoted by $\widetilde{\mathcal{B}}(v,r)$, as follows:  

\begin{description}
\item[1.] Delete from $\mathcal{B}(v,r)$ all edges incident to vertices in $[\k]$, denote by $\mathcal{B}_-(v,r)$ the resulting graph.
\item[2.] Let $\mathcal{C}_v(\mathcal{B}_-(v,r))$ be the connected component in $\mathcal{B}_-(v,r)$ that contains $v$. Add  to $\mathcal{C}_v(\mathcal{B}_-(v,r))$ all edges $(u,v)$ deleted in the previous step such that $u \in [\k]$ and $v \in \mathcal{C}_v(\mathcal{B}_-(v,r))$. The resulting graph is $\widetilde{\mathcal{B}}(v,r)$. 
\end{description}

The following is a corollary of the above.

\begin{corollary}\label{lightBallCor}
With high probability, for every vertex $v > \k_o$, $\widetilde{\mathcal{B}}(v,\om)$ belongs to one of the following categories: 
\begin{description}
\item[(i)] $\widetilde{\mathcal{B}}(v,\om)$ is a tree and all vertices are light.
\item[(ii)] $\widetilde{\mathcal{B}}(v,\om)$ has no cycles and one or two heavy vertices.
\item[(iii)] In $\widetilde{\mathcal{B}}(v,\om)$, $v$ is part of a short cycle of light vertices, and any heavy vertex in $\widetilde{\mathcal{B}}(v,\om)$ only connects to $v$ 
via edges  that are part of that cycle.
\item[(iv)] In $\widetilde{\mathcal{B}}(v,\om)$, there is a short cycle of light vertices which $v$ is not part of, which connects to $v$ through a short path $P$, and there is at most one edge $e$ on a path from from $v$ to a heavy vertex in $\widetilde{\mathcal{B}}(v,\om)$ such that $e$ is not part of $P$.
\end{description}
\end{corollary}

\subsubsection*{Degree of outer-core vertices}

For $i\in [t]$ consider the vertex $i$ and the core $[i]$. Immediately after the vertex $i$ is added,  the graph under construction at that point, ${\text{PA}_{i}(m,\d)}$, has total degree $2mi$, and $D_{i}(i)$ is a random variable taking integral value between $m$ and $2m$. We may ask, given $D_{i}(i)=a$, what is the probability that $D_i(t)=a+d$? The question can be framed as one about a Polya urn process in which the urn initially contains $a$ red balls and $2mi-a$ black balls, and the selection process has weighting functions $W_R(k)=k+\d$ and 
$W_B(k)=k-(i-1)\d$ for red and black balls respectively (see, e.g., \cite{Remco}).   

Notation: $S_i(t)$ is the sum of degrees of vertices in $[i]$ in $\PA$. The following was shown in \cite{PrefPerc}:

\begin{lemma}\label{PolyaLemma} 
There is a constant $C(m,\d)$ that depends only on $m$ and $\d$, such that for $1 \leq d \leq n \leq m(t-i)$,
\begin{equation*}
\Pr\brac{D_i(t)=a+d \mid S_i(t)-2mi=n, D_i(i)=a} \leq C(m,\d)  \frac{1}{d}\bfrac{Id}{I+n-d}^{a+\d}e^{-\frac{dI}{I+n}} 
\end{equation*}
and
\begin{equation*}
\Pr(X_R(n,a)=0) \leq  \bfrac{I}{I+n}^{a+\d},
\end{equation*}
where $I=I(i,m,\d)=i(2m+\d)-1$. 
\end{lemma}

Furthermore, the following was also given in \cite{PrefPerc}:
\begin{lemma}\label{sumConcLemma}
Suppose $\d \geq 0$ and for a vertex $i \in [t]$, $i=i(t) \rightarrow \infty$. There exists a constant $K_0>0$ that depends only on $m$ and $\d$, such that the following holds for any constant $K>K_0$ and $h$ which is smaller than a constant that depends only on $m, \d$,
\begin{equation*}
\Pr\brac{S_i(t) < \frac{1}{K}\E[S_i(t)]}  \leq e^{-hi}.
\end{equation*}
\end{lemma}

We use these lemmas to prove the following:
\begin{lemma}\label{DegreeLBLemma}
With high probability, for every $i \in [\k_o]$, $D_i(t) \geq \bfrac{t}{\k_o}^\g\frac{1}{\k_o^2}$.  
\end{lemma}
\begin{proof}
Letting $h=\frac{\log \k_o}{\k_o}$ in Lemma \ref{sumConcLemma}, we see that for some constant $K$, $S_{\k_o}(t) \geq Kt^\g \k_o^{1-\g}$. Let $z=z(t) \rightarrow \infty$ as $t \rightarrow \infty$ to be determined later. Letting $n=Kt^\g \k_o^{1-\g}-2m\k_o$ and applying  \ref{PolyaLemma},
\begin{align*}
&\Pr\brac{D_{\k_o}(t)\leq \frac{n}{\k_o z} \mid S_{\k_o}(t)-2m\k_o\geq n, D_{\k_o}(\k_o)=a} \\
&\qquad \leq \sum_{d=0}^{n/(\k_o z)} C(m,\d)  \frac{1}{d}\bfrac{Id}{I+n-d}^{a+\d}e^{-\frac{dI}{I+n}} \\
&\qquad \lesssim \bfrac{I}{I+n}^{a+\d}+ \sum_{d=1}^{n/(\k_o z)} \bfrac{I}{I+n-d}^{a+\d}d^{a+\d-1}e^{-\frac{dI}{I+n}}\\
&\qquad \leq\bfrac{I}{I+n}^{a+\d}+  \frac{I^{a+\d}}{(I+n-n/(\k_o z))^{a+\d}}\sum_{d=0}^{n/(\k_o z)} d^{a+\d-1}.
\end{align*}

Since $\k_o \rightarrow \infty$ and $z \rightarrow \infty$ as $t \rightarrow \infty$, we have $n/(\k_o z)=o(n)$, so $\frac{1}{(I+n-n/(\k_o z))^{a+\d}}\lesssim \frac{1}{(I+n)^{a+\d}}$. 

Furthermore,
\begin{equation*}
\sum_{d=0}^{n/(\k_o z)} d^{a+\d-1} \lesssim \int_0^{n/(\k_o z)} x^{a+\d-1}  \,\mathrm{d}x \leq \frac{1}{a+\d}\bfrac{n}{\k_o z}^{a+\d}.
\end{equation*}
Hence,
\begin{align*}
&\Pr\brac{D_{\k_o}(t)\leq \frac{n}{\k_o z} \mid S_{\k_o}(t)-2m\k_o\geq n, D_{\k_o}(\k_o)=a}\\
&\qquad \lesssim \bfrac{I}{I+n}^{a+\d}+ \bfrac{I}{I+n}^{a+\d}\bfrac{n}{\k_o z}^{a+\d}\\
&\qquad \lesssim \bfrac{I}{I+n}^{a+\d}+ \frac{1}{z^{a+\d}}.
\end{align*}
Now, we choose $z(t)=\k_o^2$. Then,
\begin{equation*}
\bfrac{I}{I+n}^{a+\d}=\bfrac{\k_o(2m+\d-1)}{\k_o(2m+\d-1)+  K t^\g \k_o^{1-\g}-2m\k_o}^{a+\d}
\lesssim \bfrac{\k_o}{t}^{\g(a+\d)}
=o\bfrac{1}{z^{a+\d }}
\end{equation*}
since $a \geq m \geq 5$ and $\d \geq  0$.

Thus,
\begin{equation*}
\Pr\brac{D_{\k_o}(t)\leq \frac{n}{\k_o z} \mid S_{\k_o}(t)-2m\k_o\geq n, D_{\k_o}(\k_o)=a} \lesssim \frac{1}{\k_o^{2(m+\d) }}.
\end{equation*}

A simple coupling argument shows that $D_{\k_o}(t)$ is stochastically dominated by $D_{i}(t)$ for any $i \in [\k_o]$. Therefore, taking the union bound over $[\k_o]$ we get 
$\frac{\k_o}{\k_o^{2(m+\d) }}=o(1)$ since $a \geq m \geq 5$ and $\d \geq  0$. 
\qed
\end{proof}

\section{Convergence of the majority dynamics}
In this section we show that the system converges to the initial majority opinion and bound the time it takes. Informally, Lemma \ref{treeConvLemma} shows convergence for a tree when the bias
away from $1/2$ is large enough, Lemma \ref{typicalConvergence} demonstrates for vertices outside the outer core, it only takes a constant number of steps for the probability of being red to get
below the bias threshold that Lemma \ref{treeConvLemma} requires. It also uses the fact that vertices in this range are almost tree-like. The conclusion is that there is a certain contiguous set of steps when all the vertices outside the outer core are blue. Finally, Lemma \ref{coreConvergence} shows that when this happens, the vertices in the outer core are all blue. Since there is a time
step in which all vertices are blue, the graph remains blue thereafter.

For real $p$ and integer $n>3$ define
\begin{equation}
f(n, p)= \left[\left(1+\frac{1}{\sqrt{n-1}}\right)2\right]^{\frac{2}{n-3}}4p(1-p). \label{condFunc}
\end{equation}

The following lemma was essentially first proved by the first author in \cite{LocalMaj}. Due to space restrictions, we give an informal overview here and leave the full proof in the appendix.
\begin{lemma}\label{treeConvLemma}
Let $T_u(h,d^+)$ be a depth-$h$ tree rooted at $u$ where all non-leaf vertices have degree at least $d^+ \geq 5$ odd. Let $p \in (0,\frac{1}{2})$, $k \geq 5$ and $d=k \wedge d^+$. Suppose at time $\t=0$ each vertex of $T_u(h,d^+)$ is assigned red with probability $p$.  Under $\mathcal{MP}^k$ the probability that the root $u$ is red at time step $h$ is at most $\frac{1}{4}\brac{f(d, p)}^{\brac{\frac{d-1}{2}}^h}$.
\end{lemma}
\begin{proof}[Overview]
Suppose instead of $\mathcal{MP}^k$ we had a modified version $\mathcal{MMP}^k$ on the tree in which each vertex other than the root $u$ assumes its parent is red. Under the same sequence of random
choices of which neighbours to poll, $\mathcal{MMP}^k$ can only make it more likely that $u$ ends up being red at time step $\t=h$. It also has the advantage of breaking dependencies between vertices at the same depth in the tree. Denoting $p_\t(v)$ the probability of vertex $v$ being red at time
step $\t$, we show that under $\mathcal{MMP}^k$, we get  $p=p_0(v_h)>p_1(v_{h-1})>\ldots>p_{h-1}(v_{1})>p_h(v_0)$ where  $v_i$ is a child of $v_{i-1}$ in the tree and $v_0=u$. In fact, the sequence of probabilities decays very rapidly, and we find that $p_h(v_0) < \frac{1}{4}\brac{f(d, p)}^{\brac{\frac{d-1}{2}}^h}$.
\qed
\end{proof}

\begin{lemma}\label{typicalConvergence}
Let $k \geq 5$ be odd and let $d=m \wedge k$ if $m$ is odd and $d=(m-1) \wedge k$ if $m$ is even. Let $\varepsilon$ be any positive constant, let $\t_*=B\log_d \log_d t$ where $B=B(d,\varepsilon)=\frac{1+\varepsilon}{\log_d\bfrac{d-1}{2}}$ and let $\a^*$ be the smallest positive solution for $x$ in the equation $\Pr\brac{\text{Bin}(d-1,x) \geq \frac{d-1}{2}}=x$. If each vertex in $\PA$ is red independently with probability $\a<\a^*$, then \whp\ under $\mathcal{MP}^k$ every vertex $v \in [t]\setminus[k_o]$ is blue at time steps $\t=\t_*+1, \t_*+2$. 
\end{lemma}
\begin{proof}

Let integer $n \geq 2$, $f(x)=\Pr\brac{\text{Bin}(2n,x) \geq n}$ and $g(x)=f(x)-x$. Observe $g(0)=0$ and $g(1/2)>0$.  Furthermore, $g'(x)=\binom{2n}{n}nx^{n-1}(1-x)^n-1$, whence $g'(0)=-1$. Therefore $g(x)$ has 
a root $x^*$ in $(0,1/2)$.  Now $g^{''}(x)=\binom{2n}{n}nx^{n-2}(1-x)^{n-1}\left[(n-1)-x(2n-1)\right]$  which is strictly positive on $0<x<\frac{1}{2}-\frac{1}{2(2n-1)}$ and non-positive on $\frac{1}{2}-\frac{1}{2(2n-1)}\leq x<1$. We can therefore deduce that $x^*$ is the unique root of $g(x)$ in $(0,1/2)$, and that for the interval $[c_1,c_2]$ where $0<c_1<c_2<x^*$, $g$ attains a maximum at $c_1$ or $c_2$. 
Hence, for a given $x \in [c_1,c_2]$, we have $0<f(x)<x$ and $x-f(x)=-g(x)>-\max\{g(c_1), g(c_2)\}>0$. Therefore, we need only a constant number of iterations of $f$  until $f(f(\ldots f(x))\ldots)<c_1$. When $2n=d-1$, we write $\a^*=x^*$.

Now consider a rooted tree of depth $h$ where non-leave vertices have $2n=d-1$ children, and suppose that each vertex is coloured red independently with probability $\a<\a^*$ at time $\tau=0$. By the same
argument as in the proof of Lemma \ref{treeConvLemma}, at time $\tau=1$ the depth $h-1$  vertices are red independently with probability $f(\a)< \a-c_2$. Continuing in this way, the probability the root 
is red is at most $c_1$ if $h>c_3$  where $c_3$ is a large enough finite constant. 

Let $\t'=\t_*-c_3$ and suppose that $\mathcal{B}(v,\om)$ is a tree. Since $\om=A\log\log t$ with $A$ arbitrarily large, then we may assume $\om=a\log_d \log_d t$ where $a$ is a constant such that $\om \geq \t_*+3$. This means $\mathcal{B}(w,\t_*)$ is a tree if $w$ is a neighbour of $v$ or $v$ itself. By the above, we may therefore assume that at time $t=c_3$, the depth $\t_*-c_3=\t'$ vertices are red independently with probability at most $c_1$.

Then by Lemma \ref{treeConvLemma} the probability $v$ is red at time step $\t_*$ is at most $\frac{1}{4}f(d,c_1)^{{\bfrac{d-1}{2}}^{\t'}}$. For large enough $t$,
$\t'\geq\frac{1+\varepsilon/2}{\log_d\bfrac{d-1}{2}}\log_d \log_d t$, therefore
\begin{equation*}
{\bfrac{d-1}{2}}^{\t'}\geq {\bfrac{d-1}{2}}^{\frac{1+\varepsilon/2}{\log_d\bfrac{d-1}{2}}\log_d \log_d t}=d^{(1+\varepsilon/2)\log_d \log_d t}=(\log_d t)^{1+\varepsilon/2}.
\end{equation*}
Thus, 
\begin{equation*}
f(d,c_1)^{{\bfrac{d-1}{2}}^{\t'}}\leq d^{-\log_d\bfrac{1}{f(d,c_1)}(\log_d t)^{1+\varepsilon/2}}=t^{-\log_d\bfrac{1}{f(d,c_1)}(\log_d t)^{\varepsilon/2}}.
\end{equation*}
If $f(d,c_1)<\beta<1$ where $\beta$ is a constant then the above is at most $t^{-(\log_d t)^{\varepsilon/4}}$ when $t$ is large enough. By the same logic, and since each of the children of
$v$ are also trees out to distance $\t_*$, the the same probability bound applies to them. Thus, taking the union bound, we see that all vertices $v$ such that $\mathcal{B}(v,\om)$ is a tree are
blue at times $\t_*,\t_*+1, \t_*+2$. 

We extend the above to other vertices outside $[\k_o]$. From Corollary \ref{lightBallCor}, we see that $v$ always has at most two ``bad'' edges that it can assume are always red. Since $m \geq 5$, this leaves $m-2 \geq 3$ ``good'' edges which, if they are blue, will out-vote the bad edges, regardless of what their actual colours are. Thus, suppose $e_1=(v,w_1),\ldots, e_{m-2}=(v,w_{m-2})$ are good edges. As per the proof of Lemma~\ref{treeConvLemma}, the random variables $Y_{\t}(w_i)$ for $i \in \{1,\ldots,m-2\}$ depend only on vertices in the subtree of 
$\widetilde{\mathcal{B}}(v,\om)$ for which $w_i$ is a root. This is a depth--$(\om-1)$ tree where each vertex not a leaf nor root has at least $m-1$ children. Since we may assume that $\om \geq \t_*+2$,
it follows by the above that \whp, all such $w_i$ are blue at time steps $\t_*, \t_*+1, \t_*+2$. This forces $v$ to be blue at time steps $\t_*+1, \t_*+2, \t_*+3$. Thus, we have proved that $\whp$, all vertices
$v \in [t]\setminus[\k_o]$ are blue at time steps $\t_*+1, \t_*+2$. 
\qed
\end{proof}

It remains to consider the vertices  in $[\k_o]$:
\begin{lemma}\label{coreConvergence}
If every vertex in $v \in [t]\setminus[\k_o]$ is blue at time step $\t_*+1$, then \whp\ every $v \in [\k_o]$ is blue at time step $\t_*+2$.  
\end{lemma}
\begin{proof}
Consider a vertex $v \in [\k_o]$. We partition $v$'s set of incident edges $E_v$ in $\PA$  into two sets $E_{v1}=\{(v,w) : w \in [\k_o]\}$ and $E_{v2}=E_v\setminus E_{v1}$.  
Clearly, $|E_{v1}| \leq m\k_o$, so by Lemma \ref{DegreeLBLemma}, we may assume that $|E_{v2}| \geq \bfrac{t}{\k_o}^\g\frac{1}{\k_o^2}- m\k_o$ for every $v \in [\k_0]$. Consequently, the 
probability that at time step $\t_*+1$ the majority of edges picked by $v$ are in $E_{v1}$ is zero if $d \geq 2|E_{v1}|+1$ and 
$O\brac{\Pr\brac{\text{Bin}(d,\frac{\k_o^4}{t})>\frac{d}{2}}} = O\brac{\k_o^4/t}$ if  $d \leq 2|E_{v1}|$. Taken over all vertices in $[\k_o]$ this is $o(1)$.
\qed
\end{proof}

\begin{corollary}
With high probability, $\PA$ is entirely blue at all time steps $\t \geq \t_*+2$. 
\end{corollary}

\section{Conclusion and open problems}
We have seen that with high probability, local majority dynamics on preferential attachment graphs with power law exponent at least $3$ very rapidly converges to the initial majority when the initial distribution of red vs. blue opinions is sufficiently biased away from equality. The speed of convergence is affected both by the number of neighbours polled at each step as well structural parameters of the graph, specifically, how many edges are added when a new vertex joins in the construction process of the graph. 

A natural next step would be to analyse the process for $-m<\d<0$, which generates graphs with power-law exponents between $2$ and $3$. These appear to better reflect ``real world'' networks, but our experience suggests that structural differences make the techniques of this paper ineffective in this regime.

Another direction would be to explore how adversarial placements of opinions affects outcome, as studied in \cite{ColinVoting} for random regular graphs. 

\bibliographystyle{abbrv}

\begin{thebibliography}{99}

\bibitem{LocalMaj}  M.A. Abdullah and M. Draief, Global majority consensus by local majority polling on graphs of a given degree sequence,
\emph{Discrete Applied Mathematics} 180:1--10, 2015

\bibitem{PrefPerc} M.A. Abdullah and N. Fountoulakis,  A phase transition in the evolution of bootstrap percolation
processes on preferential attachment graphs, arXiv:1404.4070 

\bibitem{ar:StatMechs} R. Albert and A.-L. Barab\'asi, Statistical mechanics of complex networks, \emph{Reviews of Modern Physics} 74:47--97, 2002.

\bibitem{AlFi} D. Aldous and J. Fill, {\em Reversible Markov Chains and
Random Walks on Graphs}, (in preparation) \url{http://stat-www.berkeley.edu/pub/users/aldous/RWG/book.html}


\bibitem{Diam} B. Bollob\'as and O. Riordan, The diameter of a scale-free random graph, \emph{Combinatorica} 24:5--34, 2004. 

\bibitem{DegSeq} B. Bollob\'as, O. Riordan, J. Spencer and G. Tusn\'ady, The degree sequence of a scale-free random graph process, 
\emph{Random Structures and Algorithms} 18:279--290, 2001

\bibitem{BuckleyOsthus2004} P.G. Buckley and D. Osthus, Popularity based random graph models leading to a scale-free degree 
sequence, \emph{Discrete Mathematics} 282:53--68, 2004

\bibitem{ColinVoting} C. Cooper, R.  Els\"{a}sser and T. Radzik, The power of two choices in distributed voting,
\emph{Proc. of The 41st International Colloquium on Automata, Languages, and Programming (ICALP)}, 2014

\bibitem{ColinCovertime}  C. Cooper and A. Frieze, The cover time of the preferential attachment graph,
\emph{Journal of Combinatorial Theory Series B} 97:269--290, 2004

\bibitem{Cruise} J. Cruise,  A. Ganesh, Probabilistic consensus via polling and majority rules,
\emph{Proc. of Allerton Conference}, 2010

\bibitem{Dor} S.N. Dorogovtsev, J.F.F. Mendes and A.N. Samukhin, Structure of growing networks with preferential linking.
{\em Physical Review Letters} 85: 4633--4636, 2000

\bibitem{Drinea} E. Drinea, M. Enachescu and M. Mitzenmacher, Variations on random graph models for the web. Technical report TR-06-01, 
Harvard University, Department of Computer Science, 2001

\bibitem{Peleg} Y. Hassin and D. Peleg, Distributed probabilistic polling and applications to proportionate agreement,
\emph{Information and Computation} 171:248--268, 2001

\bibitem{Remco}  R. van der Hofstad, {\em Random Graphs and Complex Networks}, 2013 (book available at 
\texttt{http://www.win.tue.nl/~rhofstad/NotesRGCN.pdf})

\bibitem{Mossel} E. Mossel , J. Neeman and O. Tamuz, Majority dynamics and aggregation of nformation in social networks
2012, arXiv:1207.0893 



\bibitem{Simon} H.A. Simon, On a class of skew distribution functions, \emph{ Biometrika}, 42:425--440, 1955.


\bibitem{Yule} G.U. Yule  A mathematical theory of evolution, based on the conclusions of Dr. J.G.~Willis F.R.S., 
\emph {Phil. Trans. Roy. Soc. London, B} 213:21--87, 1925  

\end{thebibliography}

\section{Appendix}
\begin{proposition}\label{binprop}
Let $N$ be a natural number and $p \in (0,\frac{1}{2})$. Then 
\begin{equation*}
\Pr\brac{\operatorname{Bin}(2N,p)\geq N} \geq \Pr\brac{\operatorname{Bin}(2N+2,p)\geq N+1}. 
\end{equation*}
\end{proposition}
\begin{proof}
Let $X$ and $Y$ be independent random variables with distributions  $\operatorname{Bin}(2N,p)$ and $\operatorname{Bin}(2,p)$ respectively, and let $Z=X+Y$. Then $\mathbf{1}_{\{X \geq N\}}=\mathbf{1}_{\{Z \geq N+1\}}$ except when  $X=N$ and  $Y=0$,  or $X=N-1$ and $Y=2$. The former case occurs with probability $p_a=\binom{2N}{N}p^N(1-p)^{N+2}$ and the latter with $p_b=\binom{2N}{N-1}p^{N+1}(1-p)^{N+1}$. Observe $p_a \geq p_b$ if and only if $p \leq \frac{N+1}{2N+1}$, which is always the case when $p < \frac{1}{2}$. 
\end{proof}

\textbf{Proof of Lemma \ref{CPCLemma} continued}
Now consider a pair of cycles $(a_1, \ldots, a_r)$ and $(b_1, \ldots, b_s)$ that share a single vertex $a_1=b=1$, and where $1 \leq r,s\leq 2\om+1$. Applying Proposition \ref{structureProp},  the expected number of such structures lying entirely in $[t]\setminus[\k]$ is bounded by
\begin{align*}
& \sum_{r=1}^{2\om+1}\sum_{s=1}^{2\om+1}\sum_{\k< a_1, \ldots, a_r}\sum_{\k< b_2, \ldots, b_s}\frac{M^{r+s}}{(a_1)^2}\prod_{i=2}^r\frac{1}{a_i}\prod_{j=2}^s\frac{1}{b_j}\\
&\lesssim \brac{\int_\k^t x^{-2}\,\mathrm{d}x} \, \sum_{r=1}^{2\om+1}\sum_{s=1}^{2\om+1}(M\log t)^{r+s}\\
& \lesssim \frac{(\log t)^{5\omega}}{\k}=o(1).
\end{align*}

Finally,  consider a cycle $a_1, \ldots, a_r$  and a connecting path $b_0, \ldots, b_\ell$  where $1 \leq r, \ell \leq 2\om+1$. Setting $b_0=a_1$ and $b_\ell=a_x$ where $a_x$ varies over the other $r-1$ vertices of the cycle, the expected number of such structures lying entirely in $[t]\setminus[\k]$ is at most

\begin{equation*}
\sum_{r=1}^{2\om+1}\sum_{\ell=1}^{2\om+1}\sum_{\k<a_1, \ldots, a_r}\sum_{\k<b_1, \ldots, b_\ell}\sum_{x=1}^r\frac{M^{r+\ell}}{(a_1a_x)^{3/2}}\prod_{\substack{i=2 \\i \neq x}}^r\frac{1}{a_i}\prod_{i=1}^{\ell-1}\frac{1}{b_j}
\lesssim \frac{(\log t)^{5\omega}}{\k}=o(1).
\end{equation*}
\qed

\textbf{Proof of Lemma \ref{typicalStructureLemma} continued}
\textbf{(ii)} Suppose the cycle  in question (which may be a self-loop or a pair of parallel edges with another vertex) has length $r$ and the path has length $\ell$. Without loss of generality, assuming $a_1=b_0=v$,
the expected number of such structures is bounded from above by
\begin{equation*}
\k \sum_{r = 1}^{2\om+1}\sum_{\ell = 1}^{\om}\sum_{\k<a_1,\ldots,a_r}\sum_{\k<b_1,\ldots,b_{\ell-1}}\frac{M^{r+\ell}}{v^{3/2}}\prod_{i=2}^{r}\frac{1}{a_i} \prod_{j=1}^{\ell-1}\frac{1}{b_j}\lesssim \frac{\k(M\log t)^{4\om}}{v^{3/2}}.
\end{equation*}
Summing over all $v>\k_0$ this is $O\bfrac{\k(M\log t)^{4\om}}{\k_o^{1/2}}=o(1)$.
\\

\textbf{(iii)} Suppose $v$ has two edges $e_1,e_2$ on short paths into $[\k]$ and that neither is on $P$. By similar reasoning to part \textbf{(i)}, there is a minimal structure $S$ with light vertices 
$a_1, \ldots, a_r$ where $r \leq 5\om$, with at most $5\om$ edges and where each light vertex has at least edges. Applying Proposition \ref{structureProp}, the expected number of structures $S$ for a given $v$ is asymptotically bounded from above by
\begin{equation*}
\k^2 M^{5\om}\sum_{r=1}^{5\om}\sum_{\k<a_1, \ldots, a_r}\frac{1}{v^{3/2}}\prod_{i=1}^r\frac{1}{a_i} \lesssim \frac{5\om\k^2 (M \log t)^{5\om}}{v^{3/2}}.
\end{equation*}
Summing over all $v>\k_0$ this is $O\bfrac{\k(M\log t)^{4\om}}{\k_o^{1/2}}=o(1)$.
\qed

\textbf{Proof of Lemma \ref{treeConvLemma}}

Let $N_v(\t)$ be the (random) set of neighbours of $v$ selected at time step $\t$. Let $d_v(k)$ be the minimum between $k$ and the largest odd number not larger the degree of $v$. Observe $|N_v(\t)|=d_v(k) \geq d$. For $v \neq u$, define $\text{Par}(v)$ to be the parent of $v$ in $T_u(h,d)$. 

We define the indicator random variables $X_0(v)=Y_0(v)=1$ if and only if vertex $v$ is coloured blue at time $\t=0$, and for $\t>0$,
\begin{equation*}
X_{\t}(v)=\mathbf{1}_{\left\{\brac{\sum_{w \in N_v(\t)}X_{\t-1}(w)} > \frac{|N_v(\t)|}{2}\right\}}
\end{equation*}
for all $v$, and 
\begin{equation*}
Y_{\t}(v)=\mathbf{1}_{\left\{\brac{\sum_{w \in N_v(\t)\setminus\{\text{Par}(v)\}}Y_{\t-1}(w)} > \frac{|N_v(\t)|}{2}\right\}}
\end{equation*}
for all $v \neq u$. Thus, $X_{\t}(v)$ represent the outcome under $\mathcal{MP}^k$ and  $Y_{\t}(v)$ is blue if and only if the number of blue children forms the majority. 

Observe $Y_{\t}(v) \leq X_{\t}(v)$ for all $v$ and $\t \geq 0$, for suppose it is the case for $\t-1$, then
\begin{eqnarray*}
Y_{\t}(v)=\mathbf{1}_{\left\{\brac{\sum_{w \in N_v(\t)\setminus\{\text{Par}(v)\}}Y_{\t-1}(w)} > \frac{|N_v(\t)|}{2}\right\}}&\leq& \mathbf{1}_{\left\{\brac{\sum_{w \in N_v(\t)\setminus\{\text{Par}(v)\}}X_{\t-1}(w)} > \frac{|N_v(\t)|}{2}\right\}}\\
 &\leq&  \mathbf{1}_{\left\{\brac{\sum_{w \in N_v(\t)}X_{\t-1}(w)} > \frac{|N_v(\t)|}{2}\right\}}=X_{\t}(v).
\end{eqnarray*}
Let $p_\t(v)=\Pr\brac{Y_\t(v)=0}$. 

We will show that $p=p_0(v_h)>p_1(v_{h-1})>\ldots>p_{h-1}(v_{1})>p_h(v_0)$ where  $v_i$ is a child of $v_{i-1}$ in the tree and $v_0=u$. In fact, the sequence of probabilities decays doubly-exponentially, with the implication that taking a union bound over many trees will still result in a value that is $o(1)$. 

Consider a vertex $v$ at depth $h-1$. Its children are leaves which are coloured red at time $\t=0$ independently with probability $p$. Therefore, 
\begin{equation*}
p_1(v) \leq \Pr\left(\operatorname{Bin}(d_v(k)-1,p)  \geq \frac{d_v(k)-1}{2} \right)\leq \Pr\left(\operatorname{Bin}(d-1,p)  \geq \frac{d-1}{2} \right)
\end{equation*}
where the second inequality follows by Proposition \ref{binprop} (see appendix). 

Since $p<\frac{1}{2}$, 
\begin{eqnarray*}
p_1(v)\leq\sum_{i=\frac{d-1}{2}}^{d-1}\binom{d-1}{i}p^i(1-p)^{d-1-i} &\leq& p^{\frac{d-1}{2}}(1-p)^{\frac{d-1}{2}}\sum_{i=\frac{d-1}{2}}^{d-1}\binom{d-1}{i}\nonumber \\
&=&p^{\frac{d-1}{2}}(1-p)^{\frac{d-1}{2}} \brac{\frac{1}{2}2^{d-1}+\frac{1}{2}\binom{d-1}{\frac{d-1}{2}}}.
\end{eqnarray*}
Using the inequality $\binom{2n}{n}\leq \frac{2^{2n}}{\sqrt{2n}}$, we have  $p_1(v) \leq \frac{1}{2}(1+\frac{1}{\sqrt{d-1}})(4p(1-p))^{\frac{d-1}{2}}$.

Let $D(x)$ denote the depth of a vertex $x$, i.e., its distance from $u$ in the tree. Observe that $\{Y_{1}(v) : D(v)=h-1\}$, $\{Y_{2}(v) : D(v)=h-2\}, \ldots ,\{Y_{h-1}(v) : D(v)=1\}$ 
are sets of independent random variables. 

Suppose that at $\t<h$ the following inequality held for all $v$ such that $D(v)=h-\t$:
\begin{equation*}
p_{\t}(v) \leq	\frac{1}{4}\left[\left(1+\frac{1}{\sqrt{d-1}}\right)2\right]^{\sum_{i=0}^{\tau-1}\brac{\frac{d-1}{2}}^i}\left(4p(1-p)\right)^{\brac{\frac{d-1}{2}}^\tau}<\frac{1}{2}, 
\end{equation*}
and define $p_\t$ to be the RHS of the above inequality. Then for $\t+1$ and all $v$ such that $D(v)=h-t-1$,
\begin{eqnarray*}
p_{\t+1}(v)&\leq&\sum_{i=\frac{d-1}{2}}^{d-1}\binom{d-1}{i}p_\t^i(1-p_\t)^{d-1-i}\\
&\leq& \frac{1}{2}\left(1+\frac{1}{\sqrt{d-1}}\right)\left(4p_\t(1-p_\t)\right)^{\frac{d-1}{2}}\\
&\leq& \frac{1}{2}\left(1+\frac{1}{\sqrt{d-1}}\right)\left(4p_\t\right)^\frac{d-1}{2}\\
&\leq& \frac{1}{2}\left(1+\frac{1}{\sqrt{d-1}}\right)\left(\left[\left(1+\frac{1}{\sqrt{d-1}}\right)2\right]^{\sum_{i=0}^{\t-1}\brac{\frac{d-1}{2}}^i}\left(4p(1-p)\right)^{\brac{\frac{d-1}{2}}^\t}\right)^\frac{d-1}{2}\\
&=& \frac{1}{4}\left[\left(1+\frac{1}{\sqrt{d-1}}\right)2\right]^{\sum_{i=0}^{\t}\brac{\frac{d-1}{2}}^i}\left(4p(1-p)\right)^{\brac{\frac{d-1}{2}}^{\t+1}}\\
&=&p_{\t+1}.
\end{eqnarray*}

Hence, for $\t \leq h$, and all $v$ such that $D(v) = h-\t$, we have
\begin{equation*}
p_\t(v) \leq \frac{1}{4}\left(\left[\left(1+\frac{1}{\sqrt{d-1}}\right)2\right]^{\frac{2}{d-3}}4p(1-p)\right)^{\brac{\frac{d-1}{2}}^\t}=\frac{1}{4}\brac{f(d, p)}^{\brac{\frac{d-1}{2}}^\t}.
\end{equation*}
\qed
\end{document}